\documentclass[reqno, 11pt]{amsart}

\usepackage[english]{babel}
\usepackage{amsmath}
\usepackage{amssymb}
\usepackage{enumerate}
\usepackage{ifthen}
\usepackage{bbm}
\usepackage{color}
\usepackage{geometry}
\provideboolean{shownotes} 
\setboolean{shownotes}{true}
\usepackage{color}
\usepackage{hyperref}
\newtheorem{theorem}{Theorem}[section]
\newtheorem{proposition}[theorem]{Proposition}
\newtheorem{lemma}[theorem]{Lemma}

\newtheorem{definition}[theorem]{Definition}

\theoremstyle{remark}
\newtheorem{remark}[theorem]{Remark}

\newcommand{\e}{\varepsilon}		       
\newcommand{\R}{\mathbb{R}}
\newcommand{\T}{\mathbb{T}^3}

\newcommand{\dive}{\mathop{\mathrm {div}}}

\newcommand{\rrho}{\sqrt{\rho}}
\newcommand{\la}{\lambda}
\newcommand{\rn}{\rho_n}
\newcommand{\rrn}{\sqrt{\rho_n}}
\newcommand{\un}{u_{n}}
\newcommand{\wn}{w_{n}}
\newcommand{\bedn}{\beta_{\delta}^{l}(u_{n})}
\newcommand{\bed}{\beta_{\delta}^{l}(u)}
\newcommand{\beln}{\bar{\beta}_{\lambda}(\rho_{n})}
\newcommand{\bel}{\bar{\beta}_{\lambda}(\rho)}
\newcommand{\belnp}{\bar{\beta}^{'}_{\lambda}(\rho_{n})}
\newcommand{\nablay}{\nabla_{y}}

\newcommand{\re}{\rho_{\e}}
\newcommand{\ue}{u_{\e}}
\newcommand{\weakto}{\rightharpoonup}
\newcommand{\weaktos}{\stackrel{*}{\rightharpoonup}}

\numberwithin{equation}{section}



\begin{document}

\title[Navier-Stokes-Korteweg Equations]{On the compactness of weak solutions to the Navier-Stokes-Korteweg equations for capillary fluids}

\author[P. Antonelli]{Paolo Antonelli}
\address[P. Antonelli]{\newline
GSSI - Gran Sasso Science Institute \\ Viale Francesco Crispi 7\\67100, L'Aquila \\Italy}
\email[]{\href{paolo.antonelli@gssi.it}{paolo.antonelli@gssi.it}}

\author[S. Spirito]{Stefano Spirito}
\address[S. Spirito]{\newline
DISIM- Dipartimento di Ingegneria e Science dell'Informazione e Matematica \\ Via Vetoio\\67100, L'Aquila \\Italy}
\email[]{\href{stefano.spirito@univaq.it}{stefano.spirito@univaq.it}}

\begin{abstract}
In this paper we consider the Navier-Stokes-Korteweg equations for a viscous compressible fluid with capillarity effects in three space dimensions. We prove compactness of finite energy weak solutions for large initial data. In contrast with previous results regarding this system, vacuum regions are allowed in the definition of weak solutions and no additional damping terms are considered. The compactness is obtained by introducing suitable truncations of the velocity field and the mass density at different scales and use only the {\em a priori} bounds obtained by the energy and the BD entropy. 
\end{abstract}

\maketitle

\section{Introduction}\label{sec:intro}
This paper is concerned about the following Navier-Stokes-Korteweg system
\begin{align}
&\partial_t\rho+\dive(\rho u)=0,\,\rho\geq 0,\label{eq:qns1}\\
&\partial_t(\rho u)+\dive(\rho u\otimes u)+\nabla\rho^{\gamma}-2\nu\dive(\rho Du)-2\kappa^2\rho\nabla\Delta\rho=0,\label{eq:qns2}
\end{align}
in a three dimensional periodic domain, so that $(t, x)\in(0, T)\times\mathbb T^3$. We endow system \eqref{eq:qns1}-\eqref{eq:qns2} with initial data
\begin{equation}\label{eq:id}
\begin{aligned}
\rho(0,x)&=\rho^0(x),\\
(\rho u)(0,x)&=\rho^0(x)u^0(x).
\end{aligned}
\end{equation} 
 The positive scalar function $\rho$ represents the density of the fluid and the three dimensional vector field $u$ is the velocity. The positive constants $\nu$ and $\kappa$, respectively, are the viscosity and the capillarity constants.\par
 The aim of this paper is to prove the compactness of solutions to \eqref{eq:qns1}-\eqref{eq:qns2}. More precisely, given a sequence of solutions to \eqref{eq:qns1}-\eqref{eq:qns2}, we show there exists a subsequence converging to a weak solution of the same system. This is one of the key steps in studying the existence of solutions for fluid dynamical systems like \eqref{eq:qns1}-\eqref{eq:qns2}, the other one being the construction of a suitable sequence of approximate solutions.

The system \eqref{eq:qns1}-\eqref{eq:qns2} falls in the class of Navier-Stokes-Korteweg equations, which in their general form read
\begin{equation}\label{eq:nskgeneral}
\begin{aligned}
&\partial_t\rho+\dive(\rho u)=0\\
&\partial_t(\rho u)+\dive(\rho u\otimes u)+\nabla p=2\nu\dive\mathbb S+2\kappa^2\dive\mathbb K,
\end{aligned}
\end{equation}
where $\mathbb S$ is the viscosity stress tensor given by
\begin{equation}\label{eq:visc}
\mathbb S=h(\rho)\,D u+g(\rho)\dive u\mathbb I,
\end{equation}
the coefficients $h$ and $g$ satisfying
\begin{equation*}
\begin{aligned}
&h\ge 0,\quad&h+3g\geq0,
\end{aligned}
\end{equation*}
 and the capillarity term $\mathbb K$ satisfies
\begin{equation}\label{eq:cap}
\dive\mathbb K=\nabla\left(\rho\dive(k(\rho)\nabla\rho)-\frac12(\rho k'(\rho)-k(\rho))|\nabla\rho|^2\right)-\dive(k(\rho)\nabla\rho\otimes\nabla\rho).
\end{equation}
The system \eqref{eq:qns1}-\eqref{eq:qns2} is then obtained from \eqref{eq:nskgeneral}-\eqref{eq:cap} by choosing  $k(\rho)=1$, $h(\rho)=\rho$ and $g(\rho)=0$. 

Systems of Korteweg type arise in modeling several physical phenomena, {\em e.g.} capillarity phenomena in fluids with diffuse interface, where the density experiences steep but still smooth change of value. $\mathbb{K}$ is called the Korteweg tensor and is derived rigorously from thermodynamic considerations by Dunn and Serrin in \cite{DS}.\par
Local existence of smooth solutions and global existence with small data for the system \eqref{eq:qns1}-\eqref{eq:qns2} have been proved in \cite{HL, HL1}. Regarding the theory of weak solutions few results are available. By exploiting some novel {\em a priori} estimates yielded by the so-called Bresch-Desjardins (BD) entropy, \cite{BD}, 
in \cite{BDL} the authors prove the global existence of weak solutions for the system \eqref{eq:qns1}-\eqref{eq:qns2}, 
by considering test functions of the type $\rho\phi$, with $\phi$ smooth and compactly supported. This particular notion of weak solutions has the advantage to avoid some mathematical difficulties which arise in the definition of the velocity field in the vacuum region.
The result was later extended in \cite{J} to the case of Quantum-Navier-Stokes, namely when we choose $k(\rho)=1/\rho$ in \eqref{eq:cap}. 
When system \eqref{eq:qns1}-\eqref{eq:qns2} is augmented by a damping term in the equation for the momentum density, then it is possible to prove the existence of global solutions by using the standard notion of weak solutions \cite{BD}. Indeed the presence of the damping term allows to define the velocity field everywhere in the domain.

However when dealing with general finite energy weak solutions to \eqref{eq:qns1}-\eqref{eq:qns2}, a major mathematical difficulty arises in defining the velocity field in the vacuum region, due to the degeneracy of the viscosity coefficient $h(\rho)=\rho$. The momentum density is always well defined, but unfortunately the standard a priori estimates given by the physical energy (and by the BD entropy) do not avoid a possible concentration which would prevent the convergence of the convective term in the compactness argument. Furthermore due to the presence of the capillarity term, a Mellet-Vasseur type estimate \cite{MV} does not seem to be available for the system \eqref{eq:qns1}-\eqref{eq:qns2}.
This problem was overcome for the quantum case when the viscosity coefficients are chosen to be $h(\rho)=\rho$ and $g(\rho)=0$. In \cite{AS}, by defining a suitable velocity it is possible to consider an alternative formulation of the system where the third order term vanishes, thus allowing the derivation of a Mellet and Vasseur type estimate for the new velocity. Alternatively, in \cite{LV} the authors replace the Mellet-Vasseur argument a by truncation method, so that they can recover the necessary compactness. In  both the results in \cite{AS} and \cite{LV} it is crucial that the viscosity and capillarity coefficients satisfy 
 \begin{equation}\label{eq:relation}
 k(\rho)=\frac{h'(\rho)^2}{\rho}. 
 \end{equation}
Note that this relation \eqref{eq:relation} plays a crucial role in the theory, see for example \cite{BGL} where the authors study the vanishing viscosity limit for the quantum Navier-Stokes equations, or \cite{AS1} where \eqref{eq:relation} is extensively exploited to construct the approximating system and \cite{BCNV} where numerical methods are performed. In particular, in all the existing literature, the relation \eqref{eq:relation} has been proved to be fundamental in order to show compactness of solutions. 
We stress that in  \eqref{eq:qns1}-\eqref{eq:qns2} the viscosity and capillarity coefficients do not satisfy the relation \eqref{eq:relation} and therefore those previous arguments cannot be applied to system under our consideration. 

In this paper we overcome this difficulty. In order to prove our compactness result, we also exploit a truncation argument. Contrarily to \cite{LV}, here it is not sufficient to truncate only the velocity field because of the lack of control on the third order term. To overcome this issue we also perform an additional truncation of the density. 
Unfortunately, this approach is not as straightforward as it would appear at a first glance. Indeed when truncating for example the convective term, some remainders cannot be simply controlled from the a priori estimates. Thus we need to introduce several scales of truncations, in order to control all the error terms.

As already remarked, inferring compactness properties for solutions to fluid dynamical systems like \eqref{eq:qns1}-\eqref{eq:qns2} are only the first step towards an existence result for global in time finite energy weak solutions. Usually this is combined with the construction of a suitable sequence of smooth approximate solutions. Potentially, this latter step could be achieved by considering the following approximating system
\begin{equation*}
\begin{aligned}
&\partial_t\re+\dive(\re\ue)=0,\\
&\partial_{t}(\re\ue)+\dive(\re\ue\otimes\ue)-2\nu\dive(\re\,D\ue)+\nabla\re^{\gamma}+\e\re|\ue|^2\ue+\e\ue\\&=2\kappa^2\re\nabla\Delta\re+2\e\re\nabla\left(\frac{\Delta\sqrt{\re}}{\sqrt{\re}}\right).
\end{aligned}
\end{equation*}
and by adapting, probably in a non trivial way, the regularisation procedure in \cite{LV} in order to rigorously derive the truncated formulation of the momentum equations. On the other hand, providing a smooth approximating system as in \cite{AS1} seems to be very challenging due to the the very rigid structure of the approximation procedure. We plan to attack this problem in future works.\par

  We conclude this introduction by describing the state of art of the analysis of the Cauchy problem  for the general system \eqref{eq:nskgeneral}-\eqref{eq:cap}.  In the case  $\kappa=0$ \eqref{eq:nskgeneral} reduces to the system of compressible Navier-Stokes equations. When the viscosity coefficient $h(\rho)$ is chosen degenerating on the vacuum region $\{\rho=0\}$ the Lions-Feireisl theory, \cite{L}, \cite{F}, and the recent approach in \cite{BJ} cannot be used because it is not possible to define the velocity in the vacuum regions. 
  The global existence of weak solutions has been proved independently in \cite{VY1} and \cite{LX} in the case $h(\rho)=\rho$ and $g(\rho)=0$. In both cases, non trivial approximation procedures are required to prove the BD entropy and the Mellet and Vasseur inequality.\par
    When the viscosity $\nu=0$, the system \eqref{eq:nskgeneral} is called Euler-Korteweg and it has been also extensively studied. In \cite{BGDD} local well-posedness has been proved, while in \cite{AH} the global existence of smooth solutions with small data has been proved. Moreover, when $k(\rho)=1/\rho$ the system \eqref{eq:nskgeneral} is called Quantum Hydrodynamic system (QHD) and arises for example in the description of quantum fluids. The global existence of finite energy weak solutions for the QHD system has been proved in \cite{AM, AM2} without restrictions on the regularity or the size of the initial data. Non uniqueness results by using convex integration methods has been proved in \cite{DFM}.\par
Moreover, relative entropy methods to study singular limits for the equations \eqref{eq:nskgeneral}-\eqref{eq:cap} have been exploited in \cite{BGL, DFM, DM1,GLT}, in particular we mention the incompressible limit in \cite{AHM} in the quantum case, the quasineutral limit \cite{DM} for the constant capillarity case and the vanishing viscosity limit in \cite{BGL}. 
  Finally, the analysis of the long time behaviour for the isothermal Quantum-Navier-Stokes equations has been performed in \cite{CCH}.

 \subsection*{Organization of the paper.}
The paper is organized as follows. In Section \ref{sec:pre} we fix the notations and give the precise definition of weak solutions of \eqref{eq:qns1}-\eqref{eq:qns2}. In Section \ref{sec:apriori} we recall the formal a priori estimates for solutions of the system \eqref{eq:qns1}-\eqref{eq:qns2}, namely the energy estimate and the BD entropy. Finally, in the Section \ref{sec:main} we prove Theorem \ref{teo:main}.

\section{Preliminaries} \label{sec:pre}
\subsection{Notations}
Given $\Omega\subset\R^3$, the space of compactly supported smooth functions with value in $\R^d$ will be $\mathcal{D}((0,T)\times\Omega;\R^{d})$. We will denote with $L^{p}(\Omega)$ the standard Lebesgue spaces and with $\|\cdot\|_{L^p}$ their norm. The Sobolev space of $L^{p}$ functions with $k$ distributional derivatives in $L^{p}$ is $W^{k,p}(\Omega)$ and in the case $p=2$ we will write $H^{k}(\Omega)$. The spaces $W^{-k,p}(\Omega)$ and $H^{-k}(\Omega)$ denote the dual spaces of $W^{k,p'}(\Omega)$ and $H^{k}(\Omega)$ where  $p'$ is the H\"older conjugate of $p$. Given a Banach space $X$ we use the the classical Bochner space for time dependent functions with value in $X$, namely $L^{p}(0,T;X)$, $W^{k,p}(0,T;X)$ and $W^{-k,p}(0,T;X)$ and when $X=L^p(\Omega)$, the norm of the space $L^{q}(0,T;L^{p}(\Omega))$ is denoted by $\|\cdot\|_{L^{q}_{t}(L^{p}_{x})}$. We denote by $Du=(\nabla u+(\nabla u)^T)/2$ the symmetric part of the gradient and by $Au=(\nabla u-(\nabla u)^T)/2$ the antisymmetric one. Finally, given a matrix  $M\in\R^{3\times 3}$ we denote by $\mathrm{Symm}\,M$, the symmetric part of $M$ and by $\mathrm{Asymm}\,M$ the antisymmetric one. 
\subsection{Definition of weak solutions and statement of the main result}
The definition of weak solution for the system \eqref{eq:qns1}-\eqref{eq:qns2} is the following
 \begin{definition}\label{def:ws}
A pair $(\rho, u )$ with $\rho\geq0$ is said to be a weak solution of the Cauchy problem \eqref{eq:qns1}-\eqref{eq:qns2}-\eqref{eq:id} if the following conditions are satisfied. 
\begin{enumerate}
\item Integrability conditions.
\begin{align*}
&\rho\in L^{\infty}(0,T;H^{1}(\T))\cap L^{2}(0,T;H^2(\T)),
&\rrho\,u\in L^{\infty}(0,T;L^{2}(\T)),\\
&\rho^{\frac{\gamma}{2}}\in L^{\infty}(0,T;L^2(\T))\cap L^{2}(0,T;H^1(\T)),
&\nabla\rrho\in L^{\infty}(0,T;L^{2}(\T)).
\end{align*}
\item Equations.\\
For any $\phi\in C_c^{\infty}([0,T);C^{\infty}(\T);\R)$.
\begin{equation*}
\int\rho^0\phi(0)\,dx+\iint\rho\phi_t+\rrho\rrho u\nabla\phi\,dxdt=0.
\end{equation*}
For any fixed $l=1,2,3$ and $\psi\in C_c^{\infty}([0,T);C^{\infty}(\T);\R)$
\begin{equation*}
\begin{aligned}
&\int \rho^0u^{0,l}\psi(0)\,dx+\iint\rrho(\rrho u^{l})\psi_t\,dxdt+\iint\rrho u\rrho u^{l}:\nabla\psi\,dxdt\\
&+\nu\iint\rrho\rrho u^{l}\Delta\psi\,dxdt+\nu\iint\rrho\rrho u\nabla\nabla_{l}\psi\,dxdt+2\nu\iint\nabla_l\rrho \rrho\,u\nabla\psi\,dxdt\\
&2\nu\iint\rrho u^{l}\nabla\rrho\nabla\psi\,dxdt
-2\iint\nabla\rho^{\frac{\gamma}{2}}\rho^{\frac{\gamma}{2}}\cdot\psi\,dxdt-2\kappa^2\iint\nabla_{l}\rho\Delta\rho\psi\,dxdt\\
&-\iint2\kappa^2\rho\Delta\rho\nabla_{l}\psi\,dxdt=0.
\end{aligned}
\end{equation*}
 \item Energy Inequality.\\
There exist $\mathcal{S}\in L^{2}((0,T)\times\T)$ such that $\rrho\mathcal{S}=\mathrm{Symm}(\nabla(\rho u))-2\nabla\rrho\otimes\rrho u)\textrm{ in }\mathcal{D}'$ and $\Lambda$ such that $\rho\,u=\rrho\Lambda$ satisfying the following energy inequality
\begin{equation*}
\begin{aligned}
&\sup_{t\in(0,T)}\int_{\T}\frac{|\Lambda(t,x)|^2}{2}+\frac{\rho(t,x)^{\gamma}}{\gamma-1}+\kappa^2|\nabla\rho(t,x)|^2\,dx+\iint|\mathcal{S}(s,x)|^2\,dxds\\
&\leq\int_{\T}\rho^0(x)|u^0(x)|^2+\frac{{\rho^{0}(x)}^{\gamma}}{\gamma-1}+\kappa^2|\nabla\rho^0(x)|^2\,dx.
\end{aligned}
\end{equation*}
\item BD Entropy.\\
There exists $\mathcal{A}\in L^{2}((0,T)\times\T)$ such that $\rrho\mathcal{A}=\mathrm{Asymm}(\nabla(\rho u))-2\nabla\rrho\otimes\rrho u)\textrm{ in }\mathcal{D}'$  such that
\begin{equation*}
\begin{aligned}
&\sup_{t\in(0,T)}\int_{\T}\frac{|\Lambda(t,x)+2\nu\nabla\rrho(t,x)|^2}{2}+\frac{\rho(t,x)^{\gamma}}{\gamma-1}+\kappa^2|\nabla\rho(t,x)|^2\,dx\\
&+\iint|\mathcal{A}(s,x)|^2\,dxds+\frac{8\nu}{\gamma}\iint|\nabla\rho^{\frac{\gamma}{2}}(s,x)|^2\,dxds+4\kappa^2\nu\iint|\Delta \rho(s,x)|^2\,dxds\\
&\leq\int_{\T}\frac{|\sqrt{\rho^0(x)}u^0(x)+2\nu\nabla\sqrt{\rho^0(x)}|^2}{2}+\frac{{\rho^{0}(x)}^{\gamma}}{\gamma-1}+\kappa^2|\nabla\rho^0(x)|^2\,dx.
\end{aligned}
\end{equation*}

\end{enumerate}
\end{definition}
\begin{remark}
Let us notice that in the case of smooth solutions we have $\sqrt{\rho}\mathcal S=\rho Du$ and 
$\sqrt{\rho}\mathcal A=\rho Au$. 
However, at present it is not clear whether, for arbitrary finite energy weak solutions, it is possible to write $\mathcal S=\sqrt{\rho}Du$, $\mathcal A=\sqrt{\rho}Au$. Indeed, from the natural bounds given in Propositions \ref{lem:s1} and \ref{lem:s3bis} below, it is straightforward to obtain that the sequences $\sqrt{\rho_n}Du_n$ and $\sqrt{\rho_n}Au_n$ are uniformly bounded in $L^2$, but further informations would be needed in order to state the convergence towards $\sqrt{\rho}Du$, $\sqrt{\rho}Au$. This important remark was already noticed in \cite{LLX} and \cite{GJX} for the Navier-Stokes equations with degenerate viscosity, then later it was also exploited in \cite{LV} for the quantum Navier-Stokes equations.
\end{remark}
In order to state our main result, we first specify the assumptions on the initial data. We consider  $\{\rho^0_n\}_{n}$ being a sequence of smooth and strictly positive functions and $\rho^0$ be a strictly positive function such that 
\begin{equation}\label{eq:hyidr}
\begin{aligned}
&\rho_n^0>0,\quad \rho_n^0\rightarrow \rho^0\textrm{ strongly in }L^{1}(\T),\\
&\{\rho_n^0\}_{n}\textrm{ is uniformly in bounded in } L^{\gamma}(\T),\\
&\{\nabla\sqrt{\rho_n^0}\}_n\textrm{ is uniformly bounded in }L^{2}(\T),\\
\end{aligned}
\end{equation}
Regarding the initial velocity, let $\{u_0^{n}\}$ be a sequence of smooth vector fields and $u^0$ be a smooth vector field such that 
\begin{equation}\label{eq:hyidu}
\begin{aligned}
&\{\sqrt{\rho_n^0}u_{n}^0\}\textrm{ is uniformly bounded in }L^{2}(\T),\\
&\rho_n^0 u_n^0\rightarrow \rho^0 u^0\textrm{ in }L^{p}(\R^3)\textrm{ with }p<2.\\ 
\end{aligned}
\end{equation}

The main theorem of this paper is the following. 
\begin{theorem}\label{teo:main}
Assume $\{\rn^0\}_n$ and $\{\rn^0\un^0\}_n$ are sequences of initial data for \eqref{eq:qns1}-\eqref{eq:qns2} satisfying \eqref{eq:hyidr} and \eqref{eq:hyidu}.
Let $\{(\rho_n, \un)\}_n$ with $\rn>0$ be a sequence of smooth solutions of \eqref{eq:qns1}-\eqref{eq:qns2} with initial data $\{\rn^0\}_n$ and $\{\rn^0\un^0\}_n$, then, up to subsequences not relabelled, there exist $(\rho, u)$  such that 
\begin{equation}\label{eq:main}
\begin{aligned}
&\rn\rightarrow\rho\textrm{ strongly in }L^{2}((0,T);H^1({\T})),\\
&\rn\un\rightarrow \rho\,u\textrm{ strongly in  }L^{p}((0,T)\times\T)\textrm{ for any }p<2,\\
\end{aligned}
\end{equation} 
and $(\rho, u)$ is a weak solutions of \eqref{eq:qns1}-\eqref{eq:qns2}-\eqref{eq:id} in the sense of Definition \ref{def:ws}.
\end{theorem}
\begin{remark}
We stress that the velocity field $u$ is not uniquely defined on the vacuum region $\{\rho=0\}$. 
\end{remark}
\begin{remark}
We stress that \eqref{eq:main} does not imply the convergence of the convective term, which on the other hand comes from the truncation arguments. 
\end{remark} 
\begin{remark}
The notion of weak solution in Definition \ref{def:ws} is weaker compared with the one in the quantum case in \cite{AS, AS1}.  Indeed, in \cite{AS,AS1} it can be proved that $\Lambda=\rrho\,u$ because
$$
\rrn\,\un\to\rrho\,u\textrm{ strongly in }L^{2}((0,T)\times\T). $$
As a consequence the energy inequality and the entire weak formulation can be written only in terms of $\rho$, $\Lambda$, $\mathcal{S}$ and $\mathcal{A}$. On the contrary, in the proof of Theorem \ref{teo:main}, we are not able to prove that $\Lambda=\rrho\,u$, but only that $\rrho\Lambda=\rho\,u$. Indeed, it is not clear whether 
$$
\rrn\,\un\weakto\rrho\,u\textrm{ weakly in }L^{2}((0,T)\times\T),$$
since we do not know that $\Lambda=0$ on $\{\rho=0\}$. 
\end{remark}

\section{A priori estimates}\label{sec:apriori}
In this section we recall the two formal {\em a priori} estimates available for solutions of \eqref{eq:qns1}-\eqref{eq:qns2}. The first lemma is the basic energy estimate for the system \eqref{eq:qns1}-\eqref{eq:qns2}. 
\begin{proposition}\label{lem:s1}
Let $(\rn,\un)$ be a smooth solution of \eqref{eq:qns1}-\eqref{eq:qns2}, then 
\begin{equation}\label{eq:en}
\begin{aligned}
\sup_{t\in (0,T)}&\left(\int\rn\frac{|\un|^2}{2}+\frac{\rn^{\gamma}}{\gamma-1}+{\kappa}^2|\nabla\rn|^2\,dx\right)+2\nu\iint\rn|D\un|^2\,dxdt\\
&=\int\rn^0\frac{|\un^0|^2}{2}+\frac{{\rn^0}^{\gamma}}{\gamma-1}+{\kappa}^2|\nabla\rn^0|^2\,dx.
\end{aligned}
\end{equation}
\end{proposition}
The second main {\em a priori} estimates is the so-called BD entropy. Although this estimate is well-known, see \cite{BD}, we give a sketch of the proof for completeness. 
\begin{proposition}\label{lem:s3bis}
Let $(\rn,\un)$ be a smooth solution of \eqref{eq:qns1}-\eqref{eq:qns2}. Then, $\wn=\un+2\nu\nabla\rn$ and $\rn$ satisfy
\begin{equation}\label{eq:bd1}
\begin{aligned}
\sup_{t\in(0,T)}&\left(\int\rn\frac{|\wn|^2}{2}+\frac{\rn^\gamma}{\gamma-1}+\kappa^2|\nabla\rn|^2\,dx\right)
+\frac{8\nu}{\gamma}\iint|\nabla\rn^{\frac{\gamma}{2}}|^2\,dxdt
+2\nu\iint\rn| A \un|^2\,dxdt\\
&+4\kappa^2\nu\iint|\Delta \rn|^2\,dxdt=\int\rn^0\frac{|\wn^0|^2}{2}+\frac{{\rn^0}^{\gamma}}{\gamma-1}+{\kappa}^2|\nabla\rn^0|^2\,dx.
\end{aligned}
\end{equation}
\end{proposition}
\begin{proof}
We first perform the effective velocity transformation. Let $c\in \R$ to be chosen later. Let us consider $\wn=\un+c\nabla\log\rn$. Then,
\begin{equation*}
\partial_t\rn+\dive(\rn\wn)=\partial_t\rn+\dive(\rn(\un+c\nabla\log\rn))=c\Delta\rn.
\end{equation*}
We recall the following elementary identities,
\begin{align*}
c(\rn\nabla\log\rn)_t&=-c\nabla\dive(\rn \un),\\
c\dive(\rn \un\otimes\nabla\log\rn+\rn\nabla\log\rn\otimes \un)&=c\Delta(\rn \un)-2c\dive(\rn D\un)\\&+c\nabla\dive(\rn \un),\\
c^{2}\dive(\rn\nabla\log\rn\otimes\nabla\log\rn)&=c^2\Delta(\rn\nabla\log\rn)\\
&-c^2\dive(\rn\nabla^2\log\rn).
\end{align*}
By using these identities it is easy to prove that 
\begin{equation*}
\begin{aligned}
\partial_t(\rn \wn)+\dive(\rn \wn\otimes \wn)&+\nabla\rn^{\gamma}-c\Delta(\rn \wn)=2(\nu\!-\!c)\dive(\rn D\wn)\\
&-(\!c^2\!\!+\!2(\nu\!-\!c)c)\dive(\rn\nabla^2\log\rn)+2\kappa^2\rn\nabla\Delta\rn.
\end{aligned}
\end{equation*}
Then, by choosing  $c=2\nu$ we obtain the following system
\begin{align}
\partial_t\rn+\dive(\rn\wn)&=2\nu\Delta\rn, \label{eq:wqns1}\\
\partial_t(\rn \wn)+\dive(\rn \wn\otimes \wn)&+\nabla\rn^{\gamma}-2\nu\Delta(\rn \wn)\nonumber\\
&+2\nu\dive(\rn D\wn)=2\kappa^2\rn\nabla\Delta\rn\label{eq:wqns2}.
\end{align}
The BD Entropy \eqref{eq:bd1} is nothing else than the energy estimate associated with the system \eqref{eq:wqns1}-\eqref{eq:wqns2}.
By multiplying \eqref{eq:wqns2} by $w_n$, by integrating in space and by using \eqref{eq:wqns1} we get 
\begin{equation}\label{eq:s4}
\frac{d}{dt}\int\rn\frac{|\wn|^2}{2}\,dx+\int\nabla\rn^{\gamma} \wn\,dx+2\nu\int\rn|A \un|^2\,dx-2\kappa^2\int\rn\nabla\Delta\rn\wn=0.
\end{equation}
Then, we multiply the \eqref{eq:wqns1} by $\gamma\rn^{\gamma-1}/(\gamma-1)$ and by integrating by parts we get 
\begin{equation}\label{eq:s5}
\frac{d}{dt}\int\frac{\rn^{\gamma}}{\gamma-1}\,dx-\int\nabla\rn^{\gamma} \wn\,dx-2\nu\gamma\int\Delta\rn\frac{\rn^{\gamma-1}}{\gamma-1}\,dx=0.
\end{equation}
Finally, by multiplying \eqref{eq:wqns1} by $-2\kappa^2\Delta\rn$ we have 
\begin{equation}\label{eq:s6}
\frac{d}{dt}\int\kappa^2|\nabla\rn|^2\,dx+4\nu\kappa^2\int|\Delta\rn|^{2}\,dx-2\kappa^2\int\dive(\rn\wn)\Delta\rn=0.
\end{equation}
By summing up \eqref{eq:s4}, \eqref{eq:s5} and \eqref{eq:s6} and integrating by parts we get \eqref{eq:bd1}.
\end{proof}
\section{Compactness}\label{sec:main}
In this Section we are going to prove the main result of our paper. 
\subsection{ Bounds independent on $n$}
First of all we collect the a priori bounds we can deduce from the Proposition \ref{lem:s1} and Proposition \ref{lem:s3bis}. By the energy estimates in Proposition \ref{lem:s1} and the assumptions \eqref{eq:hyidr}, \eqref{eq:hyidu} we have the following uniform bounds. 
\begin{equation}\label{eq:ub1} 
\begin{aligned}
\|\sqrt{\rho_n}u_n\|_{L^\infty_tL^2_x}\leq C, &\;\|\nabla\rn\|_{L^\infty_tL^2_x}\leq C,\\
\|\rho_n\|_{L^\infty_t(L^1_x\cap L^\gamma_x)}\leq C,&\; \|\sqrt{\rho_n}D u_n\|_{L^2_{t, x}}\leq C.
\end{aligned}
\end{equation}
The uniform bounds obtained by the BD Entropy, Proposition \ref{lem:s3bis}, are the following
\begin{equation}\label{eq:ub2} 
\begin{aligned}
\|\sqrt{\rho_n}w_n\|_{L^\infty_tL^2_x}\leq C,&\;\|\sqrt{\rho_n}A u_n\|_{L^2_{t, x}}\leq C,\\
\|\nabla\rho_n^{\gamma/2}\|_{L^2_{t, x}}\leq C,&\;\|\Delta\rn\|_{L^2_{t, x}}\leq C.
\end{aligned}
\end{equation}
Combining some of the bounds in \eqref{eq:ub1} and in \eqref{eq:ub2} we obtain the following bounds
\begin{equation}\label{eq:ub3} 
\begin{aligned}
\|\nabla\sqrt{\rho_n}\|_{L^\infty_tL^2_x}\leq C,&\;\|\sqrt{\rho_n}\nabla u_n\|_{L^2_{t, x}}\leq C.
\end{aligned}
\end{equation}
Of course, additional bounds can be easily obtained by interpolation and Sobolev embeddings. Here we list only the ones will be used in the sequel. By Sobolev embeddings and interpolation inequalities we get 
\begin{equation}\label{eq:ub4} 
\begin{aligned}
\|\rho_n\|_{L^2_tL^\infty_x}\leq C,\quad &
\|\nabla\rho_n\|_{L^\frac{10}{3}_{t,x}}\leq C, \quad \|\rn^{\frac{\gamma}{2}}\|_{L^{\frac{10}{3}}_{t,x}}\leq C. 
\end{aligned}
\end{equation}
By using \eqref{eq:ub1}, \eqref{eq:ub2}, \eqref{eq:ub4} and H\"older inequality we have 
\begin{equation}\label{eq:ub5}
\begin{aligned}
&\|\rn\un\|_{L^{2}_{t,x}}\leq C,\qquad&\|\nabla(\rn\un)\|_{L^{2}_{t}(L^{1}_{x})}\leq C. 
\end{aligned}
\end{equation}
Finally, by using the continuity equation \eqref{eq:qns1} we have that 
\begin{equation}\label{eq:ub6}
\|\partial_t\rn\|_{L^2_tL^1_x}\leq C.
\end{equation}
\subsection{Convergence Lemma}
By using the above uniform bounds we are now able to prove the following convergences. 
\begin{lemma}\label{lem:c1}
Let $\{(\rn, \un)\}_n$ be a sequence of solutions of \eqref{eq:qns1}-\eqref{eq:qns2}. 
\begin{enumerate}
\item Up to subsequences there exist, $\rho$, $m$, $\mathcal{S}$, $\mathcal{A}$ and $\Lambda$ such that  
\begin{align}
&\rn\rightarrow\rho\textrm{ strongly in }L^{2}(0,T;H^{1}(\T)),\label{eq:strong1}\\
&\rn\un\rightarrow m\textrm{ strongly in }L^{p}(0,T;L^{p}(\T))\textrm{ with }p\in[1,2),\label{eq:strong3}\\
&\rrn\,D(\un)\weakto \mathcal{S}\textrm{ weakly in }L^{2}((0,T)\times\T),\label{eq:weakvisc}\\
&\rrn\,A(\un)\weakto \mathcal{A}\textrm{ weakly in }L^{2}((0,T)\times\T),\label{eq:weakviscant}\\
&\rrn\un\weaktos\Lambda\textrm{ weakly* in }L^{\infty}(0,T;L^{2}(\T)).\label{eq:weakrru}
\end{align}
Moreover, $\Lambda$ is such that $\rrho\Lambda=m$. 
\item The following additional convergences hold for the density
 \begin{align}
&\nabla\rrn\weakto\nabla\rrho\textrm{ weakly in }L^{2}((0,T)\times\T),\label{eq:weaknrr}\\
&\Delta\rn\weakto\Delta\rho\textrm{ weakly in }L^{2}((0,T)\times\T),\label{eq:weakdr}\\
&\rn^{\gamma}\rightarrow\rho^{\gamma}\textrm{ strongly in }L^{1}((0,T)\times\T),\label{eq:strong4}\\
&\nabla\rn^{\frac{\gamma}{2}}\weakto\nabla\rho^{\frac{\gamma}{2}}\textrm{ weakly in }L^{2}((0,T)\times\T).\label{eq:weaknp}
\end{align}
\end{enumerate}
\end{lemma}
\begin{proof}
By using \eqref{eq:qns1} and \eqref{eq:ub5}, we have that
\begin{equation*}
\{\partial_t\rn\}_n\textrm{ is uniformly bounded in }L^{2}(0,T;H^{-1}(\T)).
\end{equation*} 
Then, since $\{\rn\}_n$ is uniformly bounded in $L^{2}(0,T;H^{2}(\T))$, by using Aubin-Lions Lemma we get \eqref{eq:strong1}. Next, by using the momentum equations and the bounds \eqref{eq:ub1}-\eqref{eq:ub5}, it is easy to prove that 
\begin{equation*}
\{\partial_t(\rn\un)\}_n\textrm{ is uniformly bounded in }L^{2}(0,T;W^{-2,\frac{3}{2}}(\T)).
\end{equation*}
Then, by using  Aubin-Lions Lemma, \eqref{eq:strong3} follows. The convergences \eqref{eq:weakvisc}, \eqref{eq:weakviscant} and \eqref{eq:weakrru} follow by standard weak compactness theorems and the equality $\rrho\Lambda=m$ follows easily from \eqref{eq:strong1} and \eqref{eq:weakrru}. Next, the convergences \eqref{eq:weaknrr}, \eqref{eq:weakdr}  follow from the the uniform bounds \eqref{eq:ub1}-\eqref{eq:ub3} and standard weak compactness arguments. Finally, The convergence \eqref{eq:strong4} is easily obtained by using \eqref{eq:strong1} and the bound \eqref{eq:ub2}, the convergence \eqref{eq:weaknp} follows by \eqref{eq:ub2} and \eqref{eq:strong1}.
\end{proof}
\begin{lemma}\label{lem:c2}
Let $f\in C\cap L^{\infty}(\R^{3};\R)$ and $(\rn, \un)$ be a solution of \eqref{eq:qns1}-\eqref{eq:qns2} and let $u$ be defined as follows: 
\begin{equation}\label{eq:defu}
u=
\left\{
\begin{array}{rll}
&\frac{m(t,x)}{\rho(t,x)}=\frac{\Lambda(t,x)}{\sqrt{\rho(t,x)}} & (t,x)\in \{\rho>0\}, \\
&0 & (t,x)\in\{\rho=0\}.
\end{array}
\right.
\end{equation}
Then, the following convergences hold. 
\begin{align}
\rn\,f(u_{n})\to\rho\,f(u)&\textrm{ strongly in }L^{p}((0,T)\times\T)\textrm{ for any }p<6,\label{eq:convro}\\
\nabla\rn\,f(u_{n})\to \nabla\rho\,f(u)&\textrm{ strongly in }L^{p}((0,T)\times\T)\textrm{ for any }p<\frac{10}{3},
\label{eq:convnro}\\
\rn\un\,f(\un)\to \rho u\,f(u)&\textrm{ strongly in }L^{p}((0,T)\times\T)\textrm{ for any }p<2,\label{eq:convm}\\
\rn^{\frac{\gamma}{2}}\,f(\un)\to \rho^{\frac{\gamma}{2}}\,f(u)&\textrm{ strongly in }L^{p}((0,T)\times\T)\textrm{ for any }p<\frac{10}{3}.\label{eq:convp}
\end{align}
\end{lemma}
\begin{proof}
We first first note that, up to a subsequence non relabelled, \eqref{eq:strong1} and \eqref{eq:strong3} imply that 
\begin{equation}\label{eq:convrom}
\begin{aligned}
&\rn\to\rho\textrm{ a.e. in }(0,T)\times\T,\\
&\rn\un\to m\textrm{ a.e. in }(0,T)\times\T,\\
&\nabla\rn\to\nabla\rho\textrm{ a.e. in }(0,T)\times\T.
\end{aligned}
\end{equation}
Moreover, by Fatou Lemma we have that 
\begin{equation}
\iint\liminf_{n\to\infty}\frac{m_n^2}{\rn}\,dxdt\leq \liminf_{n\to\infty}\iint\frac{m_n^2}{\rn}<\infty,
\end{equation}
which implies that $m=0$ on $\{\rho=0\}$ and 
\begin{equation*}
\rrho\,u\in L^{\infty}(0,T;L^{2}(\T)).
\end{equation*}
Moreover, $m=\rho\,u=\rrho\Lambda$. 
Let us prove \eqref{eq:convro}. On $\{\rho>0\}$ by using \eqref{eq:convrom} we have that 
\begin{equation*}
\rn\,f(u_{n})\to\rho\,f(u)\textrm{ a.e. in }\{\rho>0\}.
\end{equation*}
On the other hand, since $f\in L^{\infty}(\R^{3};\R)$ we have 
$$
|\rn\,f(u_{n})|\leq |\rn|\|f\|_{\infty}\to 0\textrm{ a.e. in }\{\rho=0\}. $$
Then, $\rn\,f(u_{n})\to\rho\,f(u)$ a.e. in $(0,T)\times\T$ and the convergence in \eqref{eq:convro} follows by the uniform bound 
\begin{equation*}
\|\rn\|_{L^{6}_{t,x}}\leq C
\end{equation*}
and Vitali's Theorem. 
Regarding \eqref{eq:convnro}, from Lemma \ref{lem:c1} we have that $\rho$ is a Sobolev function, then, see \cite{EG}, 
 
\begin{equation*}
\nabla\rho=0\textrm{ a.e. in }\{\rho=0\}. 
\end{equation*}
From \eqref{eq:convrom} we have that 
\begin{equation*}
\begin{aligned}
&\nabla\rn\,f(u_{n})\to\nabla\rho\,f(u)\textrm{ a.e. in }\{\rho>0\}\\
&|\nabla\rn\,f(u_{n})|\leq |\nabla\rn|\|f\|_{\infty}\to 0\textrm{ a.e. in }\{\rho=0\}.
\end{aligned}
\end{equation*}
Then, $\nabla\rn\,f(u_{n})\to\nabla\rho\,f(u)$ a.e. in $(0,T)\times\T$ and \eqref{eq:convnro} follows from the uniform bound \eqref{eq:ub4} and Vitali's Theorem. Concerning \eqref{eq:convm}, again \eqref{eq:convrom} implies the following convergences
\begin{equation*}
\begin{aligned}
&\rn\un\,f(u_{n})\to m\,f(u)\textrm{ a.e. in }\{\rho>0\},\\
&|\rn\un\,f(u_{n})|\leq |\rn\un|\|f\|_{\infty}\to 0\textrm{ a.e. in }\{\rho=0\},
\end{aligned}
\end{equation*}
which, together with \eqref{eq:ub4} and Vitali's Theorem, imply \eqref{eq:convm}. Finally, \eqref{eq:convp} follows by the same arguments used to prove \eqref{eq:convro} and the uniform bounds on the pressure in \eqref{eq:ub1} and \eqref{eq:ub2}.  
\end{proof}
\subsection{The Truncations}
Let $\bar{\beta}:\R\to\R$ be an even positive compactly supported smooth function such that 
\begin{equation*}
\bar{\beta}(z)=1\textrm{ for }z\in[-1,1],
\end{equation*}
$\mbox{supp}\,\bar{\beta}\subset (-2,2)$ and $0\leq \bar{\beta}\leq 1$. Given $\bar{\beta}$, we define 
$\tilde{\beta}:\R\to\R$ as follows: 
\begin{equation*}
\tilde{\beta}(z)=\int_{0}^{z}\bar{\beta}(s)\,ds. 
\end{equation*}
For $y\in\R^{3}$ we define for any $\delta>0$ the functions 
\begin{equation*}
\begin{aligned}
&\beta_{\delta}^{1}(y):=\frac{1}{\delta}\tilde{\beta}(\delta\, y_1)\bar{\beta}(\delta\, y_2)\bar{\beta}(\delta\, y_3),\\
&\beta_{\delta}^{2}(y):=\frac{1}{\delta}\bar{\beta}(\delta\, y_1)\tilde{\beta}(\delta\, y_2)\bar{\beta}(\delta\, y_3),\\
&\beta_{\delta}^{2}(y):=\frac{1}{\delta}\bar{\beta}(\delta\, y_1)\bar{\beta}(\delta\, y_2)\tilde{\beta}(\delta\, y_3).
\end{aligned}
\end{equation*}
Note that for fixed $l=1,2,3$ the function $\beta_{\delta}^{l}:\R^{3}\to\R$ is a truncation of the function $f(y)=y_l$. 
Finally, for any $\delta>0$  we define $\hat{\beta}_{\delta}:\R^{3}\to\R$ as
\begin{equation*}
\hat{\beta}_{\delta}(y):=\bar{\beta}(\delta\,y_1)\bar{\beta}(\delta\,y_2)\bar{\beta}(\delta\,y_3),
\end{equation*}
and for any $\la>0$ we define $\bar{\beta}_{\lambda}:\R\to\R$ as
$$\bar{\beta}_{\lambda}(s)=\bar{\beta}(\lambda\,s).$$
In the next Lemma we collect some of the main properties of $\beta_{\delta}^{l}$, $\hat{\beta}_{\delta}$ and $\bar{\beta}_{\lambda}$. Those properties are elementary and can be deduced directly from the definitions.
\begin{lemma}\label{lem:trunc}
Let $\lambda,\,\delta>0$ and $K:=\|\bar{\beta}\|_{W^{2,\infty}}$. Then, there exists $C=C(K)$ such that the following bounds hold. 
\begin{enumerate}
\item For any $\delta>0$ and $l=1,2,3$
\begin{equation}\label{eq:bed}
\begin{aligned}
&\|\beta^{l}_{\delta}\|_{L^{\infty}}\leq \frac{C}{\delta},\quad&\|\nabla\beta^{l}_{\delta}\|_{L^{\infty}}\leq{C},\quad
&\|\nabla^{2}\beta^{l}_{\delta}\|_{L^{\infty}}\leq C\,\delta,
\end{aligned}
\end{equation}
\item For any $\la>0$
\begin{equation}\label{eq:bel}
\begin{aligned}
&\|\bar{\beta}_{\la}\|_{L^{\infty}}\leq 1,\quad&\|\bar{\beta}'_{\la}\|_{L^{\infty}}\leq C\,\la,\quad&\sqrt{|s|}\bar{\beta}_{\la}(s)\leq \frac{C}{\sqrt{\la}}.
\end{aligned}
\end{equation}
\item For any $\delta>0$
\begin{equation}\label{eq:ybed}
\begin{aligned}
&&\|\hat{\beta}_{\delta}\|_{L^{\infty}}\leq 1,\quad&\|\nabla\hat{\beta}_{\delta}\|_{L^{\infty}}\leq{C\delta},\quad&|y||\hat{\beta}_{\delta}(y)|\leq \frac{C}{\delta},
\end{aligned}
\end{equation}
\item The following convergences hold for $l=1,2,3$, pointwise on $\R^{3}$, as $\delta\to 0$
\begin{equation}\label{eq:bedc}
\begin{aligned}
&\beta_{\delta}^{l}(y)\to y_{l},\quad
&(\nabla_{y}\beta_{\delta}^{l})(y)\to \nabla_{y_{l}} y,\quad
&\hat{\beta}_{\delta}(y)\to1.
\end{aligned}
\end{equation}
\item The following convergence holds pointwise on $\R$ as $\la\to 0$
\begin{equation}\label{eq:belc}
\begin{aligned}
&\bar{\beta}_{\la}(s)\to 1.
\end{aligned}
\end{equation}
\end{enumerate}
\end{lemma}
\subsection{Proof of the main Theorem}
We are now ready to prove Theorem \ref{teo:main}. 
\begin{proof}[Proof of Theorem \ref{teo:main}]
Let $(\rn,\un)$ be a solution of \eqref{eq:qns1}-\eqref{eq:qns2}. By Lemma \ref{lem:c1} there exist $\rho$, $m$, $\Lambda$ such that the convergences \eqref{eq:strong1}, \eqref{eq:strong3} and \eqref{eq:weakrru} hold. Moreover, 
by defining the velocity $u$ as in Lemma \ref{lem:c2} we have that 
\begin{equation*}
\begin{aligned}
&\rrho\,u\in L^{\infty}(0,T;L^{2}(\T),\\
&m=\rrho\Lambda=\rho\,u. 
\end{aligned}
\end{equation*}
By using \eqref{eq:strong1}, \eqref{eq:strong3} and \eqref{eq:hyidr} is straightforward to prove that 
\begin{equation*}
\int\rn^0\phi(0,x)+\iint\rn\phi_t\,dxdt+\iint\rn\un\nabla\phi\,dxdt
\end{equation*}
converges to 
\begin{equation*}
\int\rho^0\phi(0,x)+\iint\rho\phi_t\,dxdt+\iint\rho\,u\nabla\phi\,dxdt,
\end{equation*}
for any $\phi\in C^{\infty}_{c}([0,T)\times\T)$. 
Let us consider the momentum equations.  Let $l\in\{1,2,3\}$ fixed. By multiplying \eqref{eq:qns2} by $\nablay\bedn$ and by using the continuity equation \eqref{eq:qns1} we have that 
\begin{equation}\label{eq:ren1}
\begin{aligned}
&\partial_t(\rn\bedn)+\dive(\rn\un\bedn)-2\nu\dive(\rn\,D(\un))\nablay\bedn\\
&+\nabla\rn^{\gamma}\nabla_{y}\bedn-2\kappa^2\rn\nabla\Delta\rn\nablay\bedn=0.
\end{aligned}
\end{equation}
Let $\psi\in C^{\infty}_{c}([0,T)\times\T;\R)$, by multiplying \eqref{eq:ren1} by $\beln\psi$ and integrating by parts we 
get 
\begin{equation}\label{eq:ren2}
\begin{aligned}
&\int\rn^{0}\beta_{\delta}^{l}(\un^0)\bar{\beta}_{\la}(\rn^0)\psi(0,x)\,dx+\iint\rn\bedn\beln\partial_t\psi+\iint\rn\un\bedn\beln\cdot\nabla\psi\,dxdt\\
&+2\nu\iint\rrn\,D\un:\rrn\nablay\bedn\beln\otimes\nabla\psi\,dxdt-2\iint\rn^{\frac{\gamma}{2}}\nabla\rn^{\frac{\gamma}{2}}\cdot\nablay\bedn\beln\psi\,dxdt\\
&-2\kappa^2\iint\nabla\rn\Delta\rn\nablay\bedn\beln\psi\,dxdt-2\kappa^2\iint\rn\Delta\rn\nablay\bedn\beln\nabla\psi\,dxdt\\
&+\iint R^{\delta,\la}_{n}\psi\,dxdt=0.
\end{aligned}
\end{equation}
where the remainder is
\begin{equation}\label{eq:remainder}
\begin{aligned}
R^{\delta,\la}_{n}=\sum_{i=1}^{6}R^{\delta,\la}_{n,i}&=\rn\bedn\belnp\partial_t\rn\\
&+\rn\un\bedn\belnp\nabla\rn\\
\\
&-2\nu\rrn\,D\un:\rrn\nablay\bedn\otimes\nabla\rn\belnp\\
\\
&+2\kappa^2\rn\Delta\rn\nabla^{2}_{y}\bedn:\nabla\un\beln\\
\\
&+2\kappa^2\rn\Delta\rn\nablay\bedn\belnp\nabla\rn\\
\\
&-2\nu\rn\,D\un\nablay^2\bedn\nabla\un\beln.\\
\\
\end{aligned}
\end{equation}
We first perform the limit as $n$ goes to $\infty$ for $\delta$ and $\la$ fixed. Notice that, since $\bar{\beta}_{\lambda}\in L^{\infty}(\R)$, and $\{\rn\}_{n}$ converges almost everywhere, we have that
\begin{equation}\label{eq:convrot}
\beln\to\bar{\beta}_{\lambda}(\rho)\textrm{ strongly in }L^{q}((0,T)\times\T)\textrm{ for any }q<\infty.
\end{equation}
By using \eqref{eq:convro} with $p=2$ and choosing $q=2$ in \eqref{eq:convrot} we have that 
\begin{equation*}
\iint\rn\bedn\beln\partial_t\psi\,dxdt\to\iint\rho\bed\bel\partial_t\psi\,dxdt.
\end{equation*}
Next, by \eqref{eq:convm} with $p=3/2$ and choosing $q=3$ in \eqref{eq:convrot} we get
\begin{equation*}
\iint\rn\un\bedn\beln\cdot\nabla\psi\,dxdt\to \iint\rho\,u\bed\bel\cdot\nabla\psi\,dxdt.
\end{equation*}
By using \eqref{eq:weakvisc}, \eqref{eq:convro} with $p=4$ and \eqref{eq:convrot} with $q=4$ it follows
\begin{equation*}
\iint\rrn\,D\un:\rrn\nablay\bedn\beln\otimes\nabla\psi\,dxdt\to\iint\sqrt{\rho}\,\mathcal{S}:\nablay\bed\bel\otimes\nabla\psi\,dxdt.
\end{equation*}
By using \eqref{eq:weaknp}, \eqref{eq:convp} with $p=3$ and \eqref{eq:convrot} with $q=6$ it follows
\begin{equation*}
\iint\rn^{\frac{\gamma}{2}}\nabla\rn^{\frac{\gamma}{2}}\cdot\nablay\bedn\beln\psi\,dxdt\to
\iint\rho^{\frac{\gamma}{2}}\nabla\rho^{\frac{\gamma}{2}}\cdot\nablay\bed\bel\psi\,dxdt.
\end{equation*}
By using \eqref{eq:weakdr}, \eqref{eq:convnro} with $p=3$ and \eqref{eq:convrot} with $q=6$ it follows
\begin{equation*}
\iint\nabla\rn\Delta\rn\nablay\bedn\beln\psi\,dxdt\to\iint\nabla\rho\Delta\rho\nablay\bed\bel\psi\,dxdt.
\end{equation*}
Next, by using \eqref{eq:weakdr}, \eqref{eq:convro} with $p=3$ and \eqref{eq:convrot} with $q=6$ it follows
\begin{equation*}
\iint\rn\Delta\rn\nablay\bedn\beln\nabla\psi\,dxdt\to\iint\rho\Delta\rho\nablay\bed\bel\nabla\psi\,dxdt.
\end{equation*}
Finally, by using \eqref{eq:hyidr} the convergence of the term involving the initial data can be easily proved. It remains to study the remainder $R^{\delta,\la}_{n}$. We claim that there exists a $C>0$ independent on $n$, $\delta$ and $\la$ such that 
\begin{align}
&\|R^{\delta,\la}_{n}\|_{L^{1}_{t,x}}\leq C\left(\frac{\delta}{\sqrt{\la}}+\frac{\la}{\delta}+\la+\delta\right).\label{eq:r1}
\end{align}
In order to prove \eqref{eq:r1} we estimate all the terms in \eqref{eq:remainder}  separately. By using \eqref{eq:ub4}, \eqref{eq:ub6}, \eqref{eq:bed} and \eqref{eq:bel} we have
\begin{equation*}
\|R^{\delta,\la}_{n,1}\|_{L^{1}_{t,x}}\leq \|\rn\|_{L^{2}_tL^{\infty}_x}\|\partial_{t}\rn\|_{L^{2}_tL^{1}_x}\|\bedn\|_{L^{\infty}_{t,x}}
\|\belnp\|_{L^{\infty}_{t,x}}\leq C\frac{\lambda}{\delta}.
\end{equation*}
By using \eqref{eq:ub1}, \eqref{eq:ub4}, \eqref{eq:bed} and \eqref{eq:bel} it holds
\begin{equation*}
\|R^{\delta,\la}_{n,2}\|_{L^{1}_{t,x}}\leq \|\rn\un\|_{L^{2}_{t,x}}\|\nabla\rn\|_{L^{2}_{t,x}}\|\bedn\|_{L^{\infty}_{t,x}}
\|\belnp\|_{L^{\infty}_{t,x}}\leq C\frac{\lambda}{\delta}.
\end{equation*}
By using \eqref{eq:ub1}, \eqref{eq:ub4}, \eqref{eq:bed} and \eqref{eq:bel} we get 
\begin{equation*}
\|R^{\delta,\la}_{n,3}\|_{L^{1}_{t,x}}\leq \|\sqrt{\rn}\|_{L^{2}(L^{\infty})}\|\rrn\,D\un\|_{L^{2}_{t,x}}\|\nabla\rn\|_{L^{\infty}_tL^{2}_x}
\|\nablay\bedn\|_{L^{\infty}_{t,x}}\|\belnp\|_{L^{\infty}_{t,x}}\leq C\la.
\end{equation*}
By using \eqref{eq:ub1}, \eqref{eq:ub2}, \eqref{eq:bed} and \eqref{eq:bel} we have that 
\begin{equation*}
\|R^{\delta,\la}_{n,4}\|_{L^{1}_{t,x}}\leq \|\Delta\rn\|_{L^{2}_{t,x}}\|\rrn\,D\un\|_{L^2_{t,x}}\|\nablay^2\bedn\|_{L^{\infty}_{t,x}}\|\rrn\beln\|_{L^{\infty}_{t,x}}\leq C\frac{\delta}{\sqrt{\la}}.
\end{equation*}
By using 
\begin{equation*}
\|R^{\delta,\la}_{n,5}\|_{L^{1}_{t,x}}\leq \|\rn\|_{L^{2}(L^{\infty})}\|\Delta\rn\|_{L^{2}_{t,x}}\|\nabla\rn\|_{L^{\infty}(L^{2})}
\|\nablay\bedn\|_{L^{\infty}_{t,x}}\|\belnp\|_{L^{\infty}_{t,x}}\leq C\la.
\end{equation*}
Finally, by using \eqref{eq:ub1}, \eqref{eq:bed} and \eqref{eq:bel} we have
\begin{equation*}
\|R^{\delta,\la}_{n,6}\|_{L^{1}_{t,x}}\leq \|\sqrt{\rn}\nabla\un\|_{L^{2}_{t,x}}^2\|\nablay^2\bedn\|_{L^{\infty}_{t,x}}\|\beln\|_{L^{\infty}_{t.x}}\leq C\delta.
\end{equation*}
Then, \eqref{eq:r1} is proved and, when $n$ goes to infinity, we have that $(\rho, u)$ satisfies the following integral equality
\begin{equation}\label{eq:ren3}
\begin{aligned}
&\iint\rho\bed\bel\partial_t\psi+\iint\rho\,u\bed\bel\cdot\nabla\psi\,dxdt\\
&-2\nu\iint\sqrt{\rho}\mathcal{S}:\nablay\bed\bel\otimes\nabla\psi\,dxdt-\iint\rho^{\frac{\gamma}{2}}\nabla\rho^{\frac{\gamma}{2}}\cdot\nablay\bed\bel\psi\,dxdt\\
&-2\kappa^2\iint\nabla\rho\Delta\rho\nablay\bed\bel\psi\,dxdt-2\kappa^2\iint\rho\Delta\rho\nablay\bed\bel\nabla\psi\,dxdt\\
&-\int\rho^{0}\beta_{\delta}^{l}(u^0)\bar{\beta}_{\la}(\rho^0)\psi(0,x)\,dx+\langle\mu^{\delta,\la},\psi\rangle=0,
\end{aligned}
\end{equation}
where  $\mu^{\delta,\la}$ is a measure such that 
\begin{equation*}
\begin{aligned}
&R^{\delta,\la}_{n}\to\mu^{\delta,\la}\textrm{ in }\mathcal{M}(\T; \R)
\end{aligned}
\end{equation*}
and its total variations satisfies
\begin{equation}\label{eq:totalvar}
\begin{aligned}
&|\mu^{\delta,\la}|(\T)\leq C\left(\frac{\delta}{\sqrt{\la}}+\frac{\la}{\delta}+\la+\delta\right).
\end{aligned}
\end{equation}
Let $\delta=\la^{\alpha}$ with $\alpha\in (1/2,1)$, then when $\la\to0$ we have that 
\begin{equation*}
\begin{aligned}
&|\mu^{\la^{\alpha},\la}|(\T)\to 0
\end{aligned}
\end{equation*}
and by \eqref{eq:bedc}, \eqref{eq:belc} and the Lebesgue Dominated Convergence Theorem we have that \eqref{eq:ren3} converge to
 \begin{equation}\label{eq:ren4}
\begin{aligned}
&\int\rho^{0}\,u^{l,0}\psi(0,x)\,dx+\iint\rho\,u^{l}\partial_t\psi+\iint\rho\,u\,u^{l}\cdot\nabla\psi\,dxdt
-2\nu\iint\sqrt{\rho}\mathcal{S}_{lj}\nabla_{j}\psi\,dxdt\\&-\iint\rho^{\frac{\gamma}{2}}\nabla_{l}\rho^{\frac{\gamma}{2}}\psi\,dxdt
-2\kappa^2\iint\nabla_{l}\rho\Delta\rho\psi\,dxdt-2\kappa^2\iint\rho\Delta\rho\nabla_{l}\psi\,dxdt=0.
\end{aligned}
\end{equation}
Next we need to identify the tensor $\mathcal{S}$. Let $\phi\in C^{\infty}_{c}([0,T)\times\T;\R)$ and $l=1,2,3$ fixed. Then the following equality holds
\begin{equation*}
\begin{aligned}
2\iint\hat{\beta}_{\delta}(\un)\rn\,(D(\un))_{l,j}\nabla_{j}\phi\,dxdt&=\iint (\nabla(\rn \un^{l})\hat{\beta}_{\delta}(\un)\nabla\phi\,dxdt\\
&+\iint (\nabla_{l}(\rn \un))\hat{\beta}_{\delta}(\un)\nabla\phi\,dxdt\\
&-2\iint\nabla\rrn\rrn\un^{l}\hat{\beta}_{\delta}(\un)\nabla\phi\,dxdt\\
&-2\iint\nabla_l\rrn\rrn\un\hat{\beta}_{\delta}(\un)\nabla\phi\,dxdt.
\end{aligned}
\end{equation*} 
 By integrating by parts we get 
\begin{equation*}
\begin{aligned}
2\iint\rrn\hat{\beta}_{\delta}(\un)\rrn\,(D(\un))_{l,j}\nabla_{j}\phi\,dxdt=&-\iint \rn \un^{l}\hat{\beta}_{\delta}(\un)\Delta\phi\,dxdt\\
&-\iint \rn \un\hat{\beta}_{\delta}(\un)\nabla\nabla_{l}\phi\,dxdt\\
&-2\iint\nabla\rrn\rrn\un^{l}\hat{\beta}_{\delta}(\un)\nabla\phi\,dxdt\\
&-2\iint\partial_l\rrn\rrn\un\hat{\beta}_{\delta}(\un)\nabla\phi\,dxdt\\
&-\iint\bar{R}^{\delta}_{n,j}\nabla_j\phi\,dxdt,
\end{aligned}
\end{equation*}
where the remainder is
\begin{equation}\label{eq:remainder2}
\bar{R}^{\delta}_{n,j}=\rn \un^{l}\nabla_{y_k}\hat{\beta}_{\delta}(\un)\nabla_{j}\un^{k}+\rn \un^{j}\nabla_{y_k}\hat{\beta}_{\delta}(\un)\nabla_{l}\un^{k}.
\end{equation}
For fixed $\delta$, by using the convergence \eqref{eq:weakvisc} and \eqref{eq:convro} with $p=4$, we have that 
\begin{equation*}
2\iint\rrn\hat{\beta}_{\delta}(\un)\rrn\,(D(\un))_{l,j}\partial_{j}\phi\,dxdt\to2\iint\rrho\mathcal{S}_{l,j}\hat{\beta}_{\delta}(u)\partial_{j}\phi\,dxdt
\end{equation*}
Next, we have that 
\begin{equation*}
\begin{aligned}
\iint \rn \un^{l}\hat{\beta}_{\delta}(\un)\Delta\phi\,dxdt&\to\iint \rho \,u^{l}\hat{\beta}_{\delta}(u)\Delta\phi\,dxdt\\
\iint \rn \un^{j}\hat{\beta}_{\delta}(\un)\nabla^{2}_{j,l}\phi\,dxdt&\to\iint \rho\,u^{j}\hat{\beta}_{\delta}(u)\nabla^{2}_{j,l}\phi\,dxdt
\end{aligned}
\end{equation*}
because of  \eqref{eq:convm} with $p=1$. By using \eqref{eq:ybed}, \eqref{eq:convro} with $p=2$ and the weak convergence of $\nabla\rrn$ in $L^{2}_{t,x}$ we get 
\begin{equation*}
\begin{aligned}
\iint\nabla_{l}\rrn\rrn\un\hat{\beta}_{\delta}(\un)\nabla\phi\,dxdt&\to
\iint\nabla_{l}\rrho\rrho\,u\hat{\beta}_{\delta}(u)\nabla\phi\,dxdt\\
\iint\nabla\rrn\rrn\un^{l}\hat{\beta}_{\delta}(\un)\nabla\phi\,dxdt&\to
\iint\nabla\rrho\rrho\,u^{l}\hat{\beta}_{\delta}(u))\nabla\phi\,dxdt\\
\end{aligned}
\end{equation*}
Finally, by using \eqref{eq:ub1}, \eqref{eq:ub2} and \eqref{eq:ybed} we have that 
\begin{equation*}
\|\bar{R}^{\delta}_{n}\|_{L^{1}_{t,x}}\leq C\|\rrn\|_{L^{\infty}_tL^{2}_{x}}\|\rrn\,D(\un)\|_{L^{2}_{t,x}}\|\nablay\hat{\beta}_{\delta}(\un)\|_{L^{\infty}_{t,x}}\leq C\delta,
\end{equation*}
and then there exists a measure $\bar{\mu}^{\delta}$ such that 
\begin{equation}\label{eq:totvar2}
\iint\bar{R}^{\delta}_{n}\cdot\nabla\phi\,dxdt
\to\langle\bar{\mu}^{\delta},\nabla\psi\rangle,
\end{equation}
and its total variation satisfies
\begin{equation*}
|\bar{\mu}^{\delta}|(\T)\leq C\delta.
\end{equation*}
Collecting the previous convergences, we have
\begin{equation*}
\begin{aligned}
2\iint\rrho\mathcal{S}_{l,j}\hat{\beta}_{\delta}(u)\nabla_{j}\phi\,dxdt&=-\iint \rho \,u^{l}\hat{\beta}_{\delta}(u)\Delta\phi\,dxdt\\
&-\iint \rho\,u^{j}\hat{\beta}_{\delta}(u)\nabla^{2}_{j,l}\phi\,dxdt\\
&-2\iint\nabla_{l}\rrho\rrho\,u\hat{\beta}_{\delta}(u)\nabla\phi\,dxdt\\
&-2\iint\nabla\rrho\rrho\,u^{l}\hat{\beta}_{\delta}(u))\nabla\phi\,dxdt\\
&-\langle\bar{\mu}^{\delta},\nabla\psi\rangle.
\end{aligned}
\end{equation*}
Finally, by using \eqref{eq:bedc}, Dominated Convergence Theorem and \eqref{eq:totvar2} we get that 
\begin{equation*}
\begin{aligned}
2\iint\rrho\mathcal{S}_{l,j}\nabla_{j}\phi\,dxdt=&-\iint \rho \,u^{l}\Delta\phi\,dxdt
-\iint \rho\,u^{j}\nabla^{2}_{j,l}\phi\,dxdt\\
&-2\iint\nabla_{l}\rrho\rrho\,u\nabla\phi\,dxdt\\
&-2\iint\nabla\rrho\rrho\,u^{l}\nabla\phi\,dxdt.\\
\end{aligned}
\end{equation*}
By the very same arguments we identify also the tensor $\mathcal{A}$. 
Finally, the energy inequality and the BD Entropy follow from the lower semicontinuity of the norms. 
\end{proof}

\end{document}